\documentclass{article}

\usepackage{amsmath, amssymb, amsthm, amsfonts, bm}
\usepackage{enumerate}
\usepackage{graphicx, color}
\usepackage[colorlinks=true]{hyperref}
\usepackage[shortlabels]{enumitem}

\numberwithin{equation}{section}
\newtheorem{theorem}{Theorem}[section]

\newtheorem{corollary}{Corollary}[section]
\newtheorem{lemma}{Lemma}[section]
\theoremstyle{definition}
\newtheorem{definition}{Definition}[section]
\theoremstyle{remark}
\newtheorem{remark}{Remark}[section]

\newcommand{\Rnum}[1]{\uppercase\expandafter{\romannumeral #1\relax}}

\newcommand{\mr}[1]{\mathrm{#1}}
\newcommand{\mb}[1]{\mathbb{#1}}
\newcommand{\mc}[1]{\mathcal{#1}}
\newcommand\numberthis{\addtocounter{equation}{1}\tag{\theequation}}

\DeclareMathOperator{\Av}{Av} 

\def\clap#1{\hbox to 0pt{\hss#1\hss}}

\title{A multiparameter integral inequality for the dyadic maximal operator and applications}
\author{Eleftherios N. Nikolidakis}
\date{\today}

\begin{document}
\maketitle

\begin{abstract}
We prove a sharp multiparameter integral inequality for the dyadic maximal operator which refines the one-parameter inequality that is given by A. Melas in \cite{M} which in turn is applied for the evaluation of the Bellman function of two integral variables for this maximal operator. Moreover we find the exact domain of definition of the related Bellman function of three integral variables and by using the results connected with the sharpness of this new multiparameter inequality we give lower bounds of this Bellman function.
\end{abstract}

\section{Introduction} \label{sec:1}
The dyadic maximal operator on $\mb R^n$ is a useful tool in analysis and is defined by
\begin{equation} \label{eq:1p1}
\mc M_d\phi(x) = \sup\left\{ \frac{1}{|S|} \int_S |\phi(u)|\,\mr du: x\in S,\ S\subseteq \mb R^n\ \text{is a dyadic cube} \right\},
\end{equation}
for every $\phi\in L^1_\text{loc}(\mb R^n)$, where $|\cdot|$ denotes the Lebesgue measure on $\mb R^n$, and the dyadic cubes are those formed by the grids $2^{-N}\mb Z^n$, for $N=0, 1, 2, \ldots$.\\
It is well known that it satisfies the following weak type (1,1) inequality
\begin{equation} \label{eq:1p2}
\left|\left\{ x\in\mb R^n: \mc M_d\phi(x) > \lambda \right\}\right| \leq \frac{1}{\lambda} \int_{\left\{\mc M_d\phi > \lambda\right\}} |\phi(u)|\,\mr du,
\end{equation}
for every $\phi\in L^1(\mb R^n)$, and every $\lambda>0$,
from which it is easy to get the following  $L^p$-inequality
\begin{equation} \label{eq:1p3}
\|\mc M_d\phi\|_p \leq \frac{p}{p-1} \|\phi\|_p,
\end{equation}
for every $p>1$, and every $\phi\in L^p(\mb R^n)$.
It is easy to see that the weak type inequality (\ref{eq:1p2}) is the best possible. For refinements of this inequality one can consult \cite{N6}.

It has also been proved that (\ref{eq:1p3}) is best possible (see \cite{B1} and \cite{B2} for general martingales and \cite{W} for dyadic ones). An approach for studying the behaviour of this maximal operator in more depth is the introduction of the so-called Bellman functions which play the role of generalized norms of $\mc M_d$. Such functions related to the $L^p$-inequality (\ref{eq:1p3}) have been precisely identified in \cite{M}, \cite{N5} and \cite{NM}. For the study of the Bellman functions of $\mc M_d$, we use the notation $\Av_E(\psi)=\frac{1}{|E|} \int_E \psi$, whenever $E$ is a Lebesgue measurable subset of $\mb R^n$ of positive measure and $\psi$ is a real valued measurable function defined on $E$. We fix a dyadic cube  $Q$ and define the localized maximal operator $\mc M'_d\phi$ as in (\ref{eq:1p1}) but with the dyadic cubes $S$ being assumed to be contained in $Q$. Then for every $p>1$ we let
\begin{equation} \label{eq:1p4}
B_p(f,F)=\sup\left\{ \frac{1}{|Q|} \int_Q (\mc M'_d\phi)^p: \Av_Q(\phi)=f,\ \Av_Q(\phi^p)=F \right\},
\end{equation}
where $\phi$ is nonnegative in $L^p(Q)$ and the variables $f, F$ satisfy $0<f^p\leq F$. By a scaling argument it is easy to see that (\ref{eq:1p4}) is independent of the choice of $Q$ (so we may choose
$Q$ to be the unit cube $[0,1]^n$).
In \cite{M}, the function (\ref{eq:1p4}) has been precisely identified for the first time. The proof has been given in a much more general setting of tree-like structures on probability spaces.

More precisely we consider a non-atomic probability space $(X,\mu)$ and let $\mc T$ be a family of measurable subsets of $X$, that has a tree-like structure similar to the one in the dyadic case (the exact definition will be given in Section \ref{sec:2}).
Then we define the dyadic maximal operator associated to $\mc T$, by
\begin{equation} \label{eq:1p5}
\mc M_{\mc T}\phi(x) = \sup \left\{ \frac{1}{\mu(I)} \int_I |\phi|\,\mr \; d\mu: x\in I\in \mc T \right\},
\end{equation}
for every $\phi\in L^1(X,\mu)$, $x\in X$.

This operator is related to the theory of martingales and satisfies essentially the same inequalities as $\mc M_d$ does. Now we define the corresponding Bellman function of four variables (two of which are integral) of $\mc M_{\mc T}$, by
\begin{multline} \label{eq:1p6}
B_p^{\mc T}(f,F,L,k) = \sup \left\{ \int_K \left[ \max(\mc M_{\mc T}\phi, L)\right]^p\mr \; d\mu: \phi\geq 0, \int_X\phi\,\mr \; d\mu=f, \right. \\  \left. \int_X\phi^p\,\mr \; d\mu = F,\ K\subseteq X\ \text{measurable with}\ \mu(K)=k\right\},
\end{multline}
the variables $f, F, L, k$ satisfying $0<f^p\leq F $, $L\geq f$, $k\in (0,1]$.
The exact evaluation of (\ref{eq:1p6}) is given in \cite{M}, for the cases where $k=1$ or $L=f$. In the first case the author (in \cite{M}) precisely identifies the function $B_p^{\mc T}(f,F,L,1)$ by evaluating it in a first stage for the case where $L=f$. That is he precisely identifies $B_p^{\mc T}(f,F,f,1)$ (in fact $B_p^{\mc T}(f,F,f,1)=F \omega_p (\frac{f^p}{F})^p$, where                         $\omega_p: [0,1] \to [1,\frac{p}{p-1}]$ is the inverse function $H^{-1}_p$, of $H_p(z) = -(p-1)z^p + pz^{p-1}$). 

The proof of the above mentioned evaluation relies on a one-parameter integral inequality which is proved by arguments based on a linearization of the dyadic maximal operator. More precisely the author in \cite{M} proves that the inequality

\begin{equation}\label{eq:1p7}
F\geq \frac{1}{(\beta+1)^{p-1}} f^p + \frac{(p-1)\beta}{(\beta+1)^p} \int_X (M_{\mathcal{T}}\phi)^p \; d\mu,
\end{equation}
is true for every non-negative value of the parameter $\beta$ and sharp for one that depends on $f$, $F$ and $p$, namely for $\beta=\omega_p (\frac{f^p}{F})-1$. This gives as a consequence an upper bound for $B_p^{\mc T}(f,F,f,1)$, which after several technical considerations is proved to be best possible.Then  by using several calculus arguments the author in \cite{M} provides the evaluation of $B_p^{\mc T}(f,F,L,1)$ for every $L\geq f$. 

Now in \cite{NM} the authors give a direct proof of the evaluation of $B_p^{\mc T}(f,F,L,1)$ by using alternative methods. In fact they prove a sharp symmetrization principle that holds for the dyadic maximal operator, which is stated as Theorem 2.1 (see Section \ref{sec:2}).

In the second case, where $L=f$, the author (in \cite{M}) uses the evaluation of $B_p^{\mc T}(f,F,f,1)$ and provides the evaluation of the more general $B_p^{\mc T}(f,F,f,k)$, $k\in (0,1]$.

In this paper, our intention is to prove a three-parameter inequality (two of the parameters are nonnegative real numbers ordered in a specific way and the third parameter varies as an arbitrary measurable subset of the non-atomic probability space $(X,\mu)$) which generalizes and strengthens (\ref{eq:1p7}). Our aim is to use this inequality in order to study the following Bellman function problem (of three integral variables)

\begin{multline} \label{eq:1p8}
B_{p,q}^{\mc T}(f,A,F) = \sup \left\{ \int_X \left(\mc M_{\mc T}\phi\right)^p\mr \; d\mu: \phi\geq 0, \int_X\phi\,\mr \; d\mu=f, \right. \\  \left. \int_X\phi^q\,\mr \; d\mu = A,\ \int_X\phi^p\,\mr \; d\mu = F\right\},
\end{multline}
where $1<q<p$, and the variables $f,A,F$ lie in the domain of definition of the above problem. More precisely we prove an inequality that connects the $L^p-$integral of $\phi$ on $X$ and $K$, and also the $L^p-$integral of $M_{\mathcal{T}}\phi$, on $X$ and $K$, where $K$ is an arbitrary measurable subset of $X$. That is we prove the following

\begin{theorem}\label{thm:3p1}
Let $\beta\geq\gamma\geq 0$, and $K$ an arbitrary measurable subset of $X$, with measure $k\in (0,1]$. Then for every $\phi\in L^p(X,\mu)$ such that $\int_X \phi \; d\mu = f$ and $\int_X \phi^p \; d\mu = F$ the following inequality is true
\begin{multline} \label{eq:1p9}
F\geq \left[ 1-\frac{1}{(1+\gamma)^{p-1}} \right] \int_K \phi^p \; d\mu + \frac{(p-1)\beta}{(\beta+1)^p} \int_X (M_{\mathcal{T}}\phi)^p \; d\mu \\ - \frac{(p-1)\gamma}{(\beta+1)^p} \int_K (M_{\mathcal{T}}\phi)^p \; d\mu + \frac{f^p}{(\beta+1)^{p-1}}.
\end{multline}
\end{theorem}

Note that if we set $\gamma=0$ in (\ref{eq:1p9}) we get (\ref{eq:1p7}). Obviously, since (\ref{eq:1p9}) refines (\ref{eq:1p7}), we obtain that it is sharp. It is interesting to search for other forms of sharpness of (\ref{eq:1p9}), connected with the additional variables $\gamma$ and $K$ and by this way we wish to provide connections with the problem (\ref{eq:1p8}). These results are given in Section 6.  

In this paper, as in our previous ones we use combinatorial techniques as a mean to get in deeper understanding of the corresponding maximal operators. The outcome of this approach is the proof of the inequality (\ref{eq:1p9}). Moreover we prove sharpness of this inequality which gives us a motivation to reach to a lower bound for the Bellman function (\ref{eq:1p8}), in a subdomain of its domain of definition, more precisely when $f,A,F$, satisfy $A \geq \frac{f^q}{H_q(\omega_p(\frac{f^p}{F}))}$. Additionally this lower bound agrees with the function $B_{p,q}^{\mc T}(f,A,F)$, when $f,A,F$, satisfy $A= \frac{f^q}{H_q(\omega_p(\frac{f^p}{F}))}$ according to the results of \cite{DN}, and so the lower bound that we find for the Bellman function appears to be sharp as $A\rightarrow (\frac{f^q}{H_q(\omega_p(\frac{f^p}{F}))})^+ $. In the rest domain of definition of (\ref{eq:1p8}), that is when $A < \frac{f^q}{H_q(\omega_p(\frac{f^p}{F}))}$ we find a lower bound for the Bellman function (see Section 7), that also appears to be sharp as $A\rightarrow (\frac{f^q}{H_q(\omega_p(\frac{f^p}{F}))})^- $, according to the results in \cite{DN}, \cite{N5} and\cite{NM} and the proof that is presented in the last section.   \\

\bigskip

\section{Preliminaries} \label{sec:2}
Let $(X,\mu)$ be a nonatomic probability space. We give the following

\begin{definition} \label{def:2p1}
A set $\mc T$ of measurable subsets of $X$ will be called a tree if the following conditions are satisfied:
\begin{enumerate}[i)]
\item $X\in\mc T$ and for every $I\in\mc T$ we have that $\mu(I) > 0$.
\item For every $I\in\mc T$ there corresponds a finite or countable subset $C(I) \subseteq \mc T$ containing at least two elements such that
\vspace{-5pt}
\begin{enumerate}[a)]
\item the elements of $C(I)$ are pairwise disjoint subsets of $I$.
\item $I = \cup\, C(I)$.
\end{enumerate}
\item $\mc T = \cup_{\substack{ m\geq 0}} \mc T_{(m)}$, where $\mc T_{(0)} = \left\{ X \right\}$ and $\mc T_{(m+1)} = \cup_{I\in \mc T_{(m)}} C(I)$.
\item We have $\lim_{m\to\infty} \left( \sup_{I\in \mc T_{(m)}} \mu(I) \right)= 0 $
\item The tree $\mc T$ differentiates $L^1(X,\mu)$. That is for every $\phi \in L^1(X,\mu)$ it is true that \[ \lim_{\substack{x\in I \in \mc T \\ \mu (I)\to 0}} \frac{1}{\mu (I)}\int_I \phi \; d\mu = \phi(x),\] for $\mu-$almost every $x\in X$.
\end{enumerate}
\end{definition}

Then we define the dyadic maximal operator corresponding to $\mc T$ by
\begin{equation}\label{eq:2p1}
M_{\mc T}\phi (x)=sup \left\lbrace \frac{1}{\mu(I)} \int_I \mid \phi \mid \; d\mu \; : x\in I\in \mc T  \right\rbrace,
\end{equation}
for every $\phi \in L^1(X,\mu)$, $x\in X$.

We now give the following

\begin{definition}
Let $\phi: (X,\mu)\longrightarrow \mathbb{R}^{+}$. Then $\phi^{*}:(0,1]\longrightarrow \mathbb{R}^{+}$ is defined as the unique non-increasing, left continuous and equimeasurable to $\phi$ function on $(0,1]$.
\end{definition}

There are several formulas that express $\phi^{*}$, in terms of $\phi$.
One of them is as follows:
\[
\phi^{*}(t)=\inf \left( \left\lbrace y>0: \mu \left( \left\lbrace x\in X: \phi(x)>y\right\rbrace\right)<t \right\rbrace \right),
\]
for every $t\in (0,1]$. An equivalent formulation of the non increasing rearrangement can be given by
\[
\phi^{*}(t)=\sup_{e\subseteq X, \\  \mu(e)\geq t} \left[ \inf_{x\in e} \phi (x) \right],
\]
for any $t\in (0,1]$.

In \cite{NM} one can see the following symmetrization principle for the dyadic maximal operator $M_{\mc T}$.

\begin{theorem}\label{thm:2p1}
Let $g:(0,1]\longrightarrow \mathbb{R}^{+}$ be non-increasing and $G_1,G_2$ be non-decreasing and non-negative functions defined on $[0,+\infty)$. Then the following identity is true, for any $k\in (0,1]$
\begin{multline*}
\sup \left\lbrace \int_K G_1(M_{\mc T}\phi)\,G_2(\phi)\; d\mu: \phi^{*}=g \text{ and } \mu(K)=k  \right\rbrace=\\
= \int_0^k G_1\left( \frac{1}{t}\int_0^t g \right) G_2\left(g(t)\right) \; dt.
\end{multline*}
\end{theorem}

By using Theorem 2.1 one immediately can see that for the evaluation of  $B_{p,q}^{\mc T}(f,A,F)$, it is enough to find the following function (related to the Hardy transform) of three integral variables :

\begin{multline}
C_{p,q}(f,A,F) = \sup \left\{ \int_0^1 \left(\frac{1}{t}\int_0^t h\right)^p: h\geq 0, \int_0^1 h=f, \right. \\  \left. \int_0^1 h^q = A,\ \int_0^1 h^p = F\right\},
\end{multline} 
where the function $h$ varies on decreasing functions on $(0.1]$ satisfying the above mentioned conditions.

\section{A multiparameter inequality for $M_{\mc T}$} \label{sec:3}
We begin by describing a linearization of the dyadic maximal operator, as it was introduced in \cite{M}. First we give the notion of the $\mc T$-good function. Let $\phi \in L^1(X,\mu)$ be a non-negative function and for any $I\in \mc T$, set $\Av_I(\phi) = \frac{1}{\mu(I)} \int_I \phi\,\mr d\mu$. We will say that $\phi$ is $\mc T$-good, if the set
\[
\mc A_\phi = \left\{ x\in X: \mc M_{\mc T}\phi(x) > \Av_I(\phi)\ \text{for all}\ I\in \mc T\ \text{such that}\ x\in I \right\}
\]
has $\mu$-measure zero.

For example one can define, for any $m\geq 0$, and $\lambda_I\geq 0$ for each $I\in \mc T_{(m)}$ (the $m$-level of the tree $\mc T$), the following function
\[
\phi=\sum_{I\in \mc T_{(m)}} \lambda_I \chi_I,
\]
where $\chi_I$ denotes the characteristic function of $I$. It is an easy matter to show that $\phi$ is $\mc T$-good.

Suppose that we are given a $\mc T$-good function $\phi$. For any $x\in X\setminus\mc A_\phi$ (that is for $\mu$-almost all $x\in X$), we denote by $I_{\phi}(x)$ the largest element in the non empty set
\[
\big\{I\in \mc T: x\in I\ \text{and}\ \mc M_{\mc T}\phi(x) = \Av_I(\phi)\big\}.
\]

We also define for any $I\in \mc T$
\begin{align*}
A(\phi,I) &= \big\{ x\in X\setminus\mc A_\phi: I_\phi(x)=I \big\},\ \text{and we set} \\[2pt]
S_\phi &= \big\{ I\in\mc T: \mu\left(A(\phi,I)\right) > 0 \big\} \cup \big\{X\big\}.
\end{align*}
It is obvious that $\mc M_{\mc T}\phi = \sum_{I\in S_\phi} \Av_I(\phi) \chi_{A(\phi,I)}$, $\mu$-almost everywhere.

We also define the following correspondence $I \to I^\star$ with respect to $S_{\phi}$ : $I^\star$ is the smallest element of $\{ J\in S_\phi: I\subsetneq J \}$.
This is defined for every $I\in S_\phi$ except $X$. It is clear that the family of sets $\left\lbrace A(\phi,I): I\in S_{\phi}\right\rbrace$ consists of pairwise disjoint sets and it's union has full measure on $X$, since $\mu\!\left( \cup_{J\notin S_\phi} A(\phi,J) \right)=0$.

We give without proof a lemma (appearing in \cite{M}) which describes the properties of the class $S_{\phi}$, and those of the sets $A(\phi,I)$, $I\in S_{\phi}$.

\begin{lemma} \label{lem:3p1}
\begin{enumerate}[i)]
\item If $I, J\in S_\phi$ then either $A(\phi,J)\cap I = \emptyset$ or $J\subseteq I$.
\item If $I\in S_\phi$, then there exists $J\in C(I)$ such that $J\notin S_\phi$.
\item For every $I\in S_\phi$ we have that
\[
I \approx \underset{\substack{S_{\phi}\ni J \subseteq I\ \,}}{\bigcup} A(\phi,J).
\]
\item For every $I\in S_\phi$ we have that
\begin{gather*}
A(\phi,I) = I\setminus \underset{\substack{J\in S_\phi : J^\star = I\ }}{\bigcup} J,\ \ \text{and thus} \\
\mu(A(\phi,I)) = \mu(I) - \sum_{\substack{J\in S_\phi: J^\star=I\ }} \mu(J).
\end{gather*}
\end{enumerate}
\end{lemma}

Here by writing $A\approx B$, we mean that $A,B$ are measurable subsets of $X$ such that $\mu(A\setminus B)=\mu(B\setminus A)=0$.

From the above lemma we immediately get that
\[
\Av_I(\phi)=\frac{1}{\mu (I)} \sum_{\substack{J\in S_\phi: J\subseteq I\ }} \int_{A(\phi,J)} \phi \; d\mu,
\]
for any $I\in S_{\phi}$. We are now in position to prove Theorem 1.1, that is the validity of (\ref{eq:1p9}).

\begin{proof}
We begin by considering a $\mc T$-good function $\phi$, satisfying $\int_X \phi \; d\mu =f$ and $\int_X \phi^p \; d\mu =F$. Let $K$ be a measurable subset of $X$, with $\mu(K)=k\in (0,1]$ and $\beta,\gamma$ such that $\beta>\gamma>0$.

By Lemma \ref{lem:3p1} we get that $F= \int_X \phi^p \; d\mu = \sum_{\substack{I\in S_\phi}} \int_{A_I}\phi^p$, where we write $A_I$ for the set $A(\phi,I), I\in S_{\phi}$. We split the set $A_I$ in two measurable subsets $B_I,\Gamma_I$ for any $I\in S_{\phi}$, where $\mu \left( B_I \right), \mu \left( \Gamma_I \right)>0$. The choice of $B_I,\Gamma_I$ will be given in the sequel. Write $\mu (A_I)=a_I$, for $I\in S_{\phi}$. For any $I\in S_{\phi}$ we search for a constant $\tau_I>0$ for which

\begin{equation}\label{eq:3p1}
\mu (I) \, \tau_I-(\beta +1)\, \sum_{\substack{J\in S_\phi \\ J^\star=I\ }} \mu(J)- (\gamma +1)\mu(B_I)= \mu (\Gamma_I),
\end{equation}

Then (\ref{eq:3p1}) in view of Lemma \ref{lem:3p1} is equivalent to
\begin{multline} \label{eq:3p2}
\mu (I) \, \tau_I-(\beta +1)\,\left( \mu(I)-\mu(A_I) \right)- (\gamma +1)\mu(B_I)= \mu (\Gamma_I)\Leftrightarrow \\
\left[ \tau_I - (\beta +1) \right]\mu (I)+ (\beta +1)\mu(B_I)+(\beta +1)\mu(\Gamma_I)-(\gamma +1)\mu(B_I)=\mu(\Gamma_I)\Leftrightarrow \\
\left[ \tau_I - (\beta +1) \right]\mu (I)+ \beta\mu(\Gamma_I)=(\gamma-\beta)\mu(B_I),
\end{multline}
We let $\mu(\Gamma_I)=k_I\,a_I$, for some $k_I\in (0,1)$, so $\mu(B_I)=(1-k_I)a_I$. Thus (\ref{eq:3p2}) becomes
\begin{multline} \label{eq:3p3}
\left[ \tau_I - (\beta +1) \right]\mu (I)=(\gamma-\beta)(1-k_I)a_I-\beta k_I\,a_I\Leftrightarrow \\
\left[ \tau_I - (\beta +1) \right]\mu (I)= \gamma(1-k_I)a_I-\beta a_I,
\end{multline}

We now set $p_I=\frac{a_I}{\mu(I)}$, for any $I\in S_{\phi}$.
Thus (\ref{eq:3p3}) gives
\begin{multline} \label{eq:3p4}
\tau_I - (\beta +1)=\gamma(1-k_I)p_I-\beta p_I\Leftrightarrow \\
\tau_I=\left( (\beta +1)-\beta p_I \right) + (1-k_I)\gamma p_I,
\end{multline}

Note that this choice of $\tau_I, I\in S_{\phi}$, immediately gives $\tau_I>0$, since $\beta>\gamma>0$ and $0<p_I\leq 1$ for any $I\in S_{\phi}$.

We write now
\begin{multline} \label{eq:3p5}
F=\sum_{\substack{I\in S_\phi}} \int_{A_I}\phi^p\; d\mu= \sum_{\substack{I\in S_\phi}} \int_{B_I}\phi^p\; d\mu+ \sum_{\substack{I\in S_\phi}} \int_{\Gamma_I} \phi^p \; d\mu \geq \\
\geq \sum_{\substack{I\in S_\phi}} \frac{\left(\int_{B_I}\phi\; d\mu\right)^p}{\mu(B_I)^{p-1}}+\sum_{\substack{I\in S_\phi}} \frac{\left(\int_{\Gamma_I}\phi\; d\mu\right)^p}{\mu(\Gamma_I)^{p-1}},
\end{multline}
in view of H{\"o}lder's inequality. We denote the first and the second sum on the right of (\ref{eq:3p5}) by $\Sigma_1$, $\Sigma_2$ respectively.
Then by (\ref{eq:3p1}) and Lemma 3.1 iv) we obtain the following
\begin{multline} \label{eq:3p6}
\Sigma_2= \sum_{\substack{I\in S_\phi}}\frac{1}{\mu(\Gamma_I)^{p-1}}\left( \int_I \phi\; d\mu - \sum_{\substack{J\in S_\phi \\ J^\star=I\ }} \int_J \phi \; d\mu - \int_{B_I}\phi\;d\mu \right)^p=\\
=\sum_{\substack{I\in S_\phi}} \frac{\left( \mu(I)y_I-\sum_{\substack{J\in S_\phi \\ J^\star=I }} \mu(J)y_J-\int_{B_I}\phi\; d\mu \right)^p}{\left( \tau_I \mu(I)-(\beta +1)\sum_{\substack{J\in S_\phi \\ J^\star=I }}\mu(J)-(\gamma+1)\mu(B_I)\right)^{p-1}},
\end{multline}
where $y_I=Av_{I}(\phi)$, for every $I\in S_{\phi}$.
Now because of H{\"o}lder's inequality in the form
\begin{equation} \label{eq:3p7}
\frac{\left(\lambda_1+\lambda_2+...+\lambda_{\nu} \right)^p}{\left( \mu_1+\mu_2+....+\mu_{\nu} \right)^{p-1}}\leq \frac{\lambda_1^p}{\mu_1^{p-1}}+\frac{\lambda_2^p}{\mu_2^{p-1}}+...+ \frac{\lambda_{\nu}^p}{\mu_{\nu}^{p-1}},
\end{equation}
where $p>1,\mu_i>0$ and $\lambda_i\geq 0$, for $i=1,2,...,\nu$, we have in view of (\ref{eq:3p6}) that:
\begin{multline} \label{eq:3p8}
\Sigma_2\geq \sum_{\substack{I\in S_\phi}}\frac{\left( \mu(I)y_I \right)^p}{\left( \tau_I \mu(I) \right)^{p-1}}- \sum_{\substack{I\in S_\phi}}\sum_{\substack{J\in S_\phi \\ J^\star=I\ }}\frac{\left( \mu(J)y_J \right)^p}{\left( (\beta +1)\mu(J) \right)^{p-1}}-\\
-\sum_{\substack{I\in S_\phi}}\frac{1}{(\gamma+1)^{p-1}} \frac{\left( \int_{B_I}\phi\; d\mu \right)^p}{\mu(B_I)^{p-1}}.
\end{multline}
By (\ref{eq:3p8}) we obtain
\begin{multline} \label{eq:3p9}
\Sigma_1+\Sigma_2\geq \left(1-\frac{1}{(1+\gamma)^{p-1}} \right)\Sigma_1+\sum_{\substack{I\in S_\phi}}\mu(I) \frac{y_I^p}{\tau_I^{p-1}}-\sum_{\substack{I\in S_\phi \\ I\neq X}} \mu(I) \frac{y_I^p}{(\beta+1)^{p-1}}=\\
=\left(1-\frac{1}{(1+\gamma)^{p-1}} \right)\Sigma_1+\frac{y_X^p}{\tau_X^{p-1}}+\sum_{\substack{I\in S_\phi \\ I\neq X}}\mu(I) y_I^p \left( \frac{1}{\tau_I^{p-1}}-\frac{1}{(\beta+1)^{p-1}} \right)=\\
=\left(1-\frac{1}{(1+\gamma)^{p-1}} \right)\Sigma_1+\frac{f^p}{\tau_X^{p-1}}+\\
+\sum_{\substack{I\in S_\phi \\ I\neq X}} \frac{a_I}{p_I} \left( \frac{1}{\left((\beta +1-\beta p_I)+(1-k_I)\gamma p_I\right)^{p-1}} - \frac{1}{(\beta +1)^{p-1}} \right) y_I^p.
\end{multline}
Note that in (\ref{eq:3p9}) we have used the properties of the correspondence $I\longrightarrow I^{\star}$, on $S_{\phi}$.

We denote now $\Sigma_3$ the sum on the right of (\ref{eq:3p9}).
Then
\begin{equation}\label{eq:3p10}
\Sigma_3 \geq \sum_{\substack{I\in S_\phi \\ I\neq X}} \frac{1}{p_I}\left[ \frac{\beta p_I-(1-k_I)\gamma p_I}{(\beta +1)^p} (p-1) \right]a_I y_I^p,
\end{equation}
because of the inequality
\begin{equation}\label{eq:3p11}
\frac{1}{\left( (\beta +1)-s \right)^{p-1}}-\frac{1}{(\beta +1)^{p-1}}\geq\frac{(p-1)s}{(\beta +1)^p},
\end{equation}
which is true for any $\beta>0$, and $s\in[0,\beta]$, by the mean value theorem on derivatives. Note that since $\beta>\gamma>0$, we have that the quantity $s=\beta p_I-(1-k_I)\gamma p_I$ is positive and less than $\beta$ so (\ref{eq:3p11}) applies in $\Sigma_3$ , and gives (\ref{eq:3p10}). Thus
\begin{multline} \label{eq:3p12}
\Sigma_3\geq (p-1)\sum_{\substack{I\in S_\phi \\ I\neq X}} \frac{\beta -\gamma}{(\beta +1)^{p}}a_I y_I^p+ (p-1)\sum_{\substack{I\in S_\phi \\ I\neq X}}\frac{k_I \gamma}{(\beta+1)^p}a_I y_I^p=\\
=(p-1)\frac{\beta-\gamma}{(\beta +1)^p}\sum_{\substack{I\in S_\phi }}a_I y_I^p -(p-1)\frac{\beta- \gamma}{(\beta +1)^p}a_Xy_X^p+\\
+ (p-1)\frac{\gamma}{(\beta +1)^p}\sum_{\substack{I\in S_\phi }}k_I a_I y_I^p-\frac{(p-1)\gamma}{(\beta +1)^p} k_Xa_X y_X^p=\\
=(p-1)\frac{\beta-\gamma}{(\beta +1)^p}\int_X \left(M_{\mc T}\phi \right)^p\; d\mu +\\
+ (p-1)\frac{\gamma}{(\beta +1)^p}\int_{\Gamma} \left(M_{\mc T}\phi \right)^p\; d\mu-\frac{(p-1)}{(\beta +1)^p} \left( (\beta-\gamma)a_X+\gamma k_X a_X \right) f^p,
\end{multline}
where we have set $\Gamma=\bigcup_{\substack{I\in S_\phi}} \Gamma_I$.

By (\ref{eq:3p9}) and (\ref{eq:3p12}) we get
\begin{multline} \label{eq:3p13}
\Sigma_1+\Sigma_2\geq \left(1-\frac{1}{(1+\gamma)^{p-1}} \right)\Sigma_1+ (p-1)\frac{\beta-\gamma}{(\beta +1)^p}\int_X \left(M_{\mc T}\phi \right)^p\; d\mu +\\
+ (p-1)\frac{\gamma}{(\beta +1)^p}\int_{\Gamma} \left(M_{\mc T}\phi \right)^p\; d\mu + \lambda_4,
\end{multline}
where
\begin{equation}\label{eq:3p14}
\lambda_4= \frac{f^p}{\tau_X^{p-1}}-\frac{(p-1)}{(\beta +1)^p}\left( (\beta - \gamma)a_X+\gamma k_X a_X \right)f^p.
\end{equation}
By definition of $\tau_X$, (\ref{eq:3p14}) gives
\begin{multline} \label{eq:3p15}
\lambda_4= f^p \left[\frac{1}{\left((\beta+1)-\beta p_X+ (1-k_X)\gamma p_X \right)^{p-1}}-(p-1)\frac{(\beta-\gamma)a_X+\gamma a_X k_X}{(\beta +1)^p} \right]=\\
=f^p \left( \frac{1}{\left( (\beta+1)-\delta \right)^{p-1}} -(p-1) \frac{\delta}{(\beta +1)^p}\right),
\end{multline}
where $\delta=(\beta-\gamma)a_X+\gamma a_X k_X$(note that we used that $p_X=a_X$).

Now because of inequality (\ref{eq:3p11}) we have that
\[
\frac{1}{\left( (\beta+1)-s \right)^{p-1}} -(p-1) \frac{s}{(\beta +1)^p}\geq \frac{1}{(\beta +1)^{p-1}}, \forall s\in[0,\beta]
\]
and note that $\delta\in (0,\beta)$, by the definition of $\delta$.

So (\ref{eq:3p15}) gives $\lambda_4 \geq \frac{f^p}{(\beta+1)^{p-1}}$. Then, by (\ref{eq:3p13}) we have
\begin{multline} \label{eq:3p16}
\Sigma_1+\Sigma_2 \geq \left( 1-\frac{1}{(1+\gamma)^{p-1}} \right)\Sigma_1+(p-1)\frac{\beta-\gamma}{(\beta+1)^p}\int_X \left( M_{\mc T}\phi \right)^p\;d\mu +\\
+ (p-1)\frac{\gamma}{(\beta+1)^p}\int_{\Gamma} \left( M_{\mc T}\phi \right)^p\;d\mu+\frac{f^p}{(\beta+1)^{p-1}}=\\
=\left( 1-\frac{1}{(1+\gamma)^{p-1}} \right)\Sigma_1+\\
+\frac{(p-1)}{(\beta+1)^p} \left[ \beta \int_X \left( M_{\mc T}\phi \right)^p\;d\mu -\gamma \int_B \left( M_{\mc T}\phi \right)^p\;d\mu \right] + \frac{f^p}{(\beta+1)^{p-1}},
\end{multline}
where $B=\bigcup_{\substack{I\in S_\phi}}B_I=X\setminus\Gamma$.

Now (\ref{eq:3p16}) gives
\begin{multline} \label{eq:3p17}
\Sigma_2+\frac{1}{(1+\gamma)^{p-1}}\Sigma_1\geq\\
\geq \frac{f^p}{(\beta+1)^{p-1}} + \frac{(p-1)}{(\beta+1)^p} \left[ \beta \int_X \left( M_{\mc T}\phi \right)^p\;d\mu -\gamma \int_B \left( M_{\mc T}\phi \right)^p\;d\mu \right] \Rightarrow\\
\frac{1}{(1+\gamma)^{p-1}} \left( \Sigma_1 + \Sigma_2 \right)+ \left( 1-\frac{1}{(1+\gamma)^{p-1}} \right)\Sigma_2\geq\\
\geq   \frac{f^p}{(\beta+1)^{p-1}}+ \frac{(p-1)}{(\beta+1)^p} \left[ \beta \int_X \left( M_{\mc T}\phi \right)^p\;d\mu -\gamma \int_B \left( M_{\mc T}\phi \right)^p\;d\mu \right].
\end{multline}
But $F\geq\Sigma_1+\Sigma_2$, and $\Sigma_2\leq \int_\Gamma \phi^p\;d\mu$, so that we conclude from (\ref{eq:3p17}) that
\begin{multline} \label{eq:3p18}
\frac{F}{(1+\gamma)^{p-1}} +  \left( 1-\frac{1}{(1+\gamma)^{p-1}} \right) \int_\Gamma \phi^p\;d\mu \geq \frac{f^p}{(\beta+1)^{p-1}}+\\
\frac{(p-1)}{(\beta+1)^p} \left[ \beta \int_X \left( M_{\mc T}\phi \right)^p\;d\mu -\gamma \int_B \left( M_{\mc T}\phi \right)^p\;d\mu \right]
\end{multline}

Now from (\ref{eq:3p18}) we immediately get
\begin{multline} \label{eq:3p19}
\frac{1}{(1+\gamma)^{p-1}}\int_B \phi^p\; d\mu + \int_\Gamma \phi^p\;d\mu \geq\\ \geq \frac{f^p}{(\beta+1)^{p-1}}+\frac{(p-1)}{(\beta+1)^p} \left[ \beta \int_X \left( M_{\mc T}\phi \right)^p\;d\mu -\gamma \int_B \left( M_{\mc T}\phi \right)^p\;d\mu \right].
\end{multline}

But the left side of (\ref{eq:3p19}) equals $F-\left(1-\frac{1}{(1+\gamma)^{p-1}} \right) \int_B \phi^p \; d\mu$, so that (\ref{eq:3p19}) becomes
\begin{multline} \label{eq:3p20}
F\geq \left( 1-\frac{1}{(1+\gamma)^{p-1}} \right) \int_B\phi^p\;d\mu + \frac{f^p}{(\beta+1)^{p-1}} +\\
+ \frac{(p-1)\beta}{(\beta+1)^p} \int_X \left( M_{\mc T}\phi \right)^p\;d\mu - \frac{(p-1)\gamma}{(\beta+1)^p} \int_B \left( M_{\mc T}\phi \right)^p\;d\mu.
\end{multline}

Inequality (\ref{eq:3p20}) is in fact true for every choice of $B$, since every measurable subset $B$, of $X$ can be written as  $B=\bigcup_{\substack{I\in S_\phi}}B_I$, where $B_I=B\cap A_I$. Then setting $\Gamma_I=A_I\setminus B_I$ and following the above proof, we obtain the validity of (\ref{eq:3p20}). Theorem \ref{thm:3p1} is thus proved for any $\phi$ which is $\mc T$-good function (replace $B$ by $K$). Note that in the above proof we have used the fact that $\mu(B_I)>0$, for every $I\in S_{\phi}$, but this can be applied (by using the fact that $(X, \mu)$ is nonatomic) to prove (\ref{eq:3p20}) even if $\mu(B_I)=0$, for some $I\in S_{\phi}$. Now if $\phi\in L^p(X,\mu)$ is arbitrary, we consider the sequence $(\phi_m)_m$, where $\phi_m=\sum_{\substack{J\in \mc T_{(m)}}} Av_J(\phi)\, \chi_J$, and we set $\Phi_m=\sum_{\substack{J\in \mc T_{(m)}}} \max \left\lbrace Av_I(\phi): J\subseteq I \in \mc T \right\rbrace \chi_I $.

Then since $Av_J(\phi)=Av_I(\phi_m)$, for any $J\in \mc T$ for which $J\subseteq I \in \mc T_{(m)}$, we immediately see that $\Phi_m= M_{\mc T}\phi_m$.

Obviously $\int_X \phi_m \; d\mu=\int_X \phi \; d\mu=f $, and we can easily see that $F_m=\int_X \phi_m^p\;d\mu \leq \int_X \phi^p \;d\mu=F$. That is $\phi_m\in L^p (X,\mu), \forall m\in \mathbb{N}$.

Additionally $\Phi_m$ converges monotonically to $M_{\mc T}\phi$. Now $\phi_m$ is $\mc T$-good for any $m\in \mathbb{N}$, so that (\ref{eq:3p20}) is true, for $\phi_m$ , and for any $B\subseteq X$ measurable. Since $M_{\mc T}\phi_m $ increases to $ M_{\mc T}\phi$ on $X$, we get
\[ \lim_m \int_X \left(M_{\mc T}\phi_m \right)^p\;d\mu = \int_X \left( M_{\mc T}\phi\right)^p \; d\mu, \] and
\[ \lim_m \int_B \left( M_{\mc T}\phi_m \right)^p\; d\mu =\int_B \left( M_{\mc T}\phi \right)^p\;d\mu, \] while by the construction of $\phi_m$, and the fact that the tree $\mc T$ differentiates $L^1(X,\mu)$ we obtain that $\phi_m\longrightarrow \phi$, $\mu$-a.e on $X$. Now since $\phi_m\leq M_{\mc T}\phi_m\leq M_{\mc T}\phi$ and $M_{\mc T}\phi\in L^p(X,\mu)$(because $\phi \in L^p(X,\mu)$), we have, using the dominated convergence theorem that $\lim_m \int_X \phi_m^p=F$ and $\lim_m \int_B \phi_m^p=\int_B \phi^p \; d\mu$. From all these facts we deduce the validity of (\ref{eq:3p20}) for general $\phi\in L^p(X,\mu)$. For $\beta=\gamma>0$,(\ref{eq:3p20}) remains true by continuity reasons.
\end{proof}
\section{A first application of Theorem 1.1} \label{sec:4}

Let $h:(0,1]\longrightarrow \mathbb{R}^{+}$ be an arbitrary non-increasing function such that $\int_0^1 h=f$ and $\int_0^1 h^p=F$, where the variables $f$, $F$ satisfy $f^p<F$. Let $k\in (0,1]$ and fix a non atomic probability space $(X,\mu)$, equipped with a tree structure $\mc T$, such that $\mc T$ differentiates $L^p(X,\mu)$. By the proof of Theorem \ref{thm:2p1} (see \cite{NM}), we can construct a family $(\phi_{\alpha})_{\alpha\in (0,1]}$, of non-negative measurable functions defined on $(X,\mu)$, and a family $(K_{\alpha})_{\alpha\in (0,1]}$ of measurable subsets of $X$, such that the following hold: $\phi_{\alpha}^{\star}=h, \forall \alpha\in (0,1]$, $\lim_{\alpha\rightarrow 0^{+}}\int_{K_{\alpha}}\left(M_{\mc T}\phi_{\alpha} \right)^p\; d\mu=\int_0^k \left( \frac{1}{t}\int_0^t h \right)^p \;dt$, $\lim_{\alpha\rightarrow 0^{+}}\int_{K_{\alpha}} \phi_{\alpha}^p\;d\mu = \int_0^k h^p$ and $\lim_{\alpha\rightarrow 0^{+}} \mu(K_{\alpha})=k$.
If we apply the inequality (\ref{eq:1p9}), for $\phi_{\alpha}$ and $K_{\alpha}$, for any $\alpha\in (0,1]$, we get:

\begin{multline} \label{eq:4p1}
F \geq \left( 1-\frac{1}{(1+\gamma)^{p-1}} \right) \int_{K_{\alpha}} \phi_{\alpha}^p\;d\mu + \frac{f^p}{(\beta+1)^{p-1}}+\\
+\frac{(p-1)\beta}{(\beta+1)^p} \int_X \left( M_{\mc T}\phi_{\alpha} \right)^p\;d\mu - \frac{(p-1)\gamma}{(\beta+1)^p} \int_{K_{\alpha}} \left( M_{\mc T}\phi_{\alpha} \right)^p\;d\mu,
\end{multline}
for any $\beta\geq \gamma>0$.

Obviously $\int_X \phi_{\alpha}\; d\mu=f$ and $\int_X \phi_{\alpha}^p \; d\mu=F$, since $\phi_{\alpha}^{\star}=h, \forall \alpha\in (0,1]$.
Letting $\alpha\rightarrow 0^{+}$, we immediately see by (\ref{eq:4p1}) that
\begin{multline} \label{eq:4p2}
\frac{(p-1)\beta}{(\beta+1)^p} \int_0^1\left(\frac{1}{t} \int_0^t h \right)^p\; dt \leq \frac{(p-1)\gamma}{(\beta+1)^p}\int_0^k\left(\frac{1}{t} \int_0^t h \right)^p\; dt +F - \frac{f^p}{(\beta+1)^{p-1}}+\\
+ \left( \frac{1}{(1+\gamma)^{p-1}}-1 \right)\int_0^k h^p.
\end{multline}

Set now $\delta=\delta_k=\left( \frac{\int_0^k\left(\frac{1}{t}\int_0^t h\right)^p\; dt}{\int_0^k h^p} \right)^{\frac{1}{p}}$. Obviously $1\leq \delta< \frac{p}{p-1}$ and $\delta=1\Leftrightarrow h$ is constant on $(0,k]$. We assume that $\beta>\delta-1>0$.
We wish, for any such $\beta$, to minimize the right side of (\ref{eq:4p2}), with respect to $\gamma\in(0,\beta)$. 

For this purpose we define
\[
G_{\beta}(\gamma)= \frac{(p-1)\gamma}{(\beta+1)^p} \int_0^k\left(\frac{1}{t} \int_0^t h \right)^p\; dt + \frac{1}{(1+\gamma)^{p-1}}\int_0^k h^p,
\]
for $\gamma\in  (0,\beta]$. Note that
\[
G'_{\beta}(\gamma)= \frac{(p-1)}{(\beta+1)^p} \int_0^k\left(\frac{1}{t} \int_0^t h \right)^p\; dt - \frac{(p-1)}{(\gamma+1)^p}\int_0^k h^p.
\]
We note at this point that if we fix $0<\beta\leq\delta-1$, then by the calculations that follow, the right side of (\ref{eq:4p2}) is minimized when $\gamma=0$, thus producing the respective inequality to (\ref{eq:1p7}) for the Hardy transform operator (by applying Theorem 2.1), and this does not give us further information for the integral properties of the Hardy transform of $h$. Thus we choose to work on the range $\beta>\delta-1>0$.    

By the definition of $G_{\beta}$ we get: $G'_{\beta}(\gamma)=0 \Leftrightarrow \frac{\beta+1}{\gamma+1}=\delta \Leftrightarrow \gamma = \frac{\beta+1}{\delta}-1$.
Since $\beta > \delta-1$, if we set $\gamma_0=\frac{\beta+1}{\delta}-1$ we have that $\gamma_0\in (0,\beta)$. We easily get now that $\min \left\lbrace G_{\beta}(\gamma): \gamma\in (0,\beta]\right\rbrace = G_{\beta}(\gamma_0)$. Replacing the value $\gamma_0$ into (\ref{eq:4p2}) for any $\beta> \delta-1$, and using the definition of $\delta$ we get

\begin{multline} \label{eq:4p3}
\int_0^1\left(\frac{1}{t} \int_0^t h \right)^p\; dt \leq \frac{1}{\beta} \left( \frac{\beta+1}{\delta}-1\right)\delta^p\int_0^k h^p+ \frac{(\beta+1)^p}{(p-1)\beta}F- \frac{(\beta+1)}{(p-1)\beta}f^p+\\
+ \frac{(\beta+1)^p}{(p-1)\beta}\left( -1 +\left(\frac{\delta}{\beta+1} \right)^{p-1} \right)\int_0^k h^p,
\end{multline}
$\forall \beta>\delta-1$.

Now the right side of (\ref{eq:4p3}), equals
\begin{multline*}
\frac{(\beta+1)}{\beta(p-1)}\left( p\delta^{p-1} \int_0^k h^p -f^p \right)-\\
-\left(\frac{\beta+1}{\beta}-1 \right)\delta^p \int_0^k h^p+ \frac{(\beta+1)^p}{(p-1)\beta}\left(F-\int_0^k h^p \right)=\\
= \frac{(\beta+1)}{\beta(p-1)} \left(p\delta^{p-1}\int_0^k h^p- (p-1)\delta^p \int_0^k h^p-f^p \right)+\\
+\delta^p \int_0^k h^p + \frac{(\beta+1)^p}{(p-1)\beta}\int_k^1 h^p=\\
= \frac{(\beta+1)}{\beta(p-1)} \left(H_{p}(\delta) \int_0^k h^p -f^p\right)+ \delta^p \int_0^k h^p + \frac{(\beta+1)^p}{(p-1)\beta} \int_k^1 h^p=\\
=\int_0^k \left(\frac{1}{t} \int_0^t h \right)^p\; dt + \Lambda(\beta)
\end{multline*}
where
\begin{equation}\label{eq:4p4}
\Lambda(\beta)=\frac{(\beta+1)^p}{(p-1)\beta}\int_k^1 h^p + \frac{(\beta+1)}{(p-1)\beta} \left(H_{p}(\delta) \int_0^k h^p -f^p\right)
\end{equation}

Assume at this point that $\delta$ satisfies
\begin{equation} \label{eq:4p5}
\delta=\delta_k\leq \omega_p\left( \frac{f^p}{F}\right),
\end{equation}
while we also assume that $\int_0^k h^p<F$, that is $\int_k^1 h^p > 0$.
We wish to find the infimum value of $\Lambda(\beta)$, for $\beta>\delta-1$, when $\delta$ satisfies (\ref{eq:4p5}).

It is a simple matter to show that
\[
\Lambda'(\beta)= -\frac{H_p (\beta+1)}{(p-1)\beta^2}\int_k^1 h^p - \frac{1}{(p-1)\beta^2}\left(H_{p}(\delta) \int_0^k h^p -f^p\right).
\]

We solve now the equation $\Lambda'(\beta)=0 \Leftrightarrow$
\begin{equation}\label{eq:4p6}
H_p(\beta+1)=\frac{f^p-H_{p}(\delta) \int_0^k h^p }{\int_k^1 h^p}.
\end{equation}
Note that the right side of (\ref{eq:4p6}) is less or equal than $H_{p}(\delta)$, that is
\begin{equation} \label{eq:4p7}
H_p(\delta)\geq \frac{f^p-H_{p}(\delta) \int_0^k h^p }{\int_k^1 h^p}
\end{equation}

Indeed (\ref{eq:4p7}) is equivalent to $H_p(\delta)F\geq f^p \Leftrightarrow \delta \leq \omega_p \left(\frac{f^p}{F} \right)$, which is true in view of the assumption that we made on $\delta$.

Now $H_p$, defined on $[1,+\infty)$ satisfies the following:
$H_p(1)=1$, $H_p$ is strictly decreasing, and $\lim_{\substack{x\rightarrow +\infty}} H_p(x)=-\infty$.

Thus there exists a unique value $\beta_0 \geq\delta-1$ for which, we have equality in (\ref{eq:4p6}). That is
\begin{equation} \label{eq:4p8}
H_p(\beta_0+1)= \frac{f^p-H_{p}(\delta) \int_0^k h^p }{\int_k^1 h^p}.
\end{equation}

As is easily seen for this value of $\beta_0$ we have that $\inf_{\substack{\delta-1<\beta<+\infty}} \Lambda(\beta)=\Lambda(\beta_0)$, thus (\ref{eq:4p3}) and (\ref{eq:4p4}) give in view of the above calculations that
\begin{equation} \label{eq:4p9}
\int_k^1 \left( \frac{1}{t} \int_0^t h \right)^p \;dt \leq \frac{(\beta_0+1)^p}{(p-1)\beta_0}\int_k^1 h^p + \frac{\beta_0+1}{(p-1)\beta_0}\left(H_{p}(\delta) \int_0^k h^p-f^p\right).
\end{equation}

It is not difficult now to show, that the right side of (\ref{eq:4p9}) equals
\[
\int_k^1 h^p \omega_p \left(\frac{f^p-H_{p}(\delta) \int_0^k h^p }{\int_k^1 h^p}\right)^p, \ \ \text{where } \omega_p:(-\infty,1]\longrightarrow [1,+\infty)
\]
is the inverse of $H_p: H_p^{-1}$. Thus (\ref{eq:4p9}) states that for any $h:(0,1]\longrightarrow\mathbb{R}^{+}$ non-increasing, with $\int_0^1 h = f, \int_0^1 h^p = F$ and any $k\in(0,1]$ for which $1<\delta_k\leq \omega_p \left(\frac{f^p}{F} \right)$ and $\int_k^1 h^p > 0$, we have:
\begin{equation}\label{eq:4p10}
\int_k^1 \left( \frac{1}{t} \int_0^t h \right)^p \;dt \leq \int_k^1 h^p \omega_p \left(\frac{f^p-H_{p}(\delta) \int_0^k h^p }{\int_k^1 h^p}\right)^p.
\end{equation}
It is now easily seen that there exists a function $g_1:(0,1]\rightarrow\mathbb{R}^+$ satisfying the following properties:
$\int_0^1 g_1 = f, \int_0^1 g_1^p = F$ and $\frac{1}{t} \int_0^t g_1=\omega_p \left(\frac{f^p}{F} \right)g_1(t), \forall t\in(0,1]$. In fact we set $g_1(t):=\frac{f}{\alpha}t^{-1+\frac{1}{\alpha}}$, where $\alpha$ is given by $\alpha=\omega_p \left(\frac{f^p}{F}\right)$. 
Note that (\ref{eq:4p10}) is sharp since if we consider the function $h=g_1$ we get for any $k\in(0,1]$ that $\delta_k=\omega_p \left(\frac{f^p}{F} \right)$ and then the right side of (\ref{eq:4p10}) equals:

\begin{multline*}
\int_k^1 h^p \omega_p  \left( \frac{f^p-H_{p}(\delta) \int_0^k h^p }{\int_k^1 h^p} \right)^p
=\int_k^1 h^p \omega_p  \left( \frac{f^p}{F}\frac{ F-\int_0^k h^p }{\int_k^1 h^p} \right)^p=\\
=\int_k^1 h^p \omega_p\left(\frac{f^p}{F} \right)^p=\int_k^1 \left(\frac{1}{t} \int_0^t h \right)^p\; dt
\end{multline*}
so we have equality in (\ref{eq:4p10}) for this choice of $h$. Using Theorem 2.1, the sharpness of (\ref{eq:4p10}), and the calculus arguments that are given right above we conclude, by choosing $\beta=\delta-1$ and letting $\gamma$ tend to zero, the sharpness of inequality (\ref{eq:1p9}), for any $k\in(0,1]$.

Let now $h:(0,1]\longrightarrow\mathbb{R}^{+}$ be a non-increasing function, satisfying $\int_0^1 h = f$ and $ \int_0^1 h^p = F$, then the set of $k$'s belonging on $(0,1)$ for which $\delta_k\leq \omega_p \left(\frac{f^p}{F} \right)$ is a non empty subset of $(0,1]$.

This is obviously true for $h=g_1$, while if $h\neq g_1$ we have that $$\int_0^1 \left(\frac{1}{t} \int_0^t h \right)^p\; dt < F \omega_p \left(\frac{f^p}{F} \right)^p,$$ because $g_1$ is the unique non-increasing function on $(0,1]$ for which we get $\int_0^1 g_1 = f$, $\int_0^1 g_1^p = F$ and $\int_0^1 \left( \frac{1}{t} \int_0^t g_1 \right)^p \;dt = F \omega_p \left(\frac{f^p}{F} \right)^p$, (see the results of \cite{N8}).

Thus considering such a $k\in(0,1]$, we get by (\ref{eq:4p10}), that for any $h:(0,1]\longrightarrow\mathbb{R}^{+}$ non-increasing with $\int_0^1 h = f$, the
inequality
\begin{equation}\label{eq:4p11}
H_p(\delta^{'}_k)\geq \left(\frac{f^p-H_{p}(\delta_k) \int_0^k h^p }{\int_k^1 h^p}\right)
\end{equation}
is true, where $\delta^{'}_k=\left(\frac{\int_k^1 \left(\frac{1}{t} \int_0^t h \right)^p\; dt}{\int_k^1 h^p}\right)^{\frac{1}{p}}$, or that
\begin{equation}\label{eq:4p12}
  H_p(\delta_k)\int_0^k h^p + H_p(\delta^{'}_k)\int_k^1 h^p \geq f^p,
\end{equation}
for any $h$ and $k\in (0,1]$ as above.

Inequality (\ref{eq:4p12}) (or equivalently (\ref{eq:4p10})) and its sharpness gives us even more information for the geometric behaviour of $\mc M_{\mc T}$ because of the appearance of the free parameter $k\in(0,1]$ which is invoked under the condition $\delta_k\leq \omega_p \left(\frac{f^p}{F} \right)$. We turn now to applications of (\ref{eq:1p9}) to the study of the Bellman function $B_{p,q}^{\mc T}(f,A,F)$. For this purpose we first find the domain of definition of this Bellman function as is presented right below.  
\section{The range of $(f,A,F)$ } \label{sec:5}
Let $(X,\mu)$ be a probability space. We prove the following
\begin{lemma} \label{lem:5p1}
Let $\phi \in L^p(X,\mu)$ be non-negative and $(f,A,F)$ be such that 
$\int_X \phi \; d\mu=f$, $\int_X \phi^q \; d\mu=A$ and $\int_X \phi^p \; d\mu=F$ where the indices $1,q,p$ satisfy $1<q<p$. Then the following inequalities are satisfied
\begin{equation}\label{eq:5p1}
f^q \leq A\leq f^{\frac{p-q}{p-1}}F^{\frac{q-1}{p-1}}
\end{equation}
\end{lemma}

\begin{proof}
Since $(X,\mu)$ is a probability space the first inequality that is stated in (\ref{eq:5p1}) is an immediate consequence of Holder's inequality. For the proof of the second one we write
\begin{equation}\label{eq:5p2}
\int_X \phi^q \; d\mu = \int_X \phi^{q\lambda} \phi^{q(1-\lambda)} \; d\mu 
\end{equation}
where $\lambda$ is defined by $\lambda=\frac{p-q}{q(p-1)}$. Obviously the following inequality is true: $0<\lambda < \frac{1}{q}$. Then if we define $a=\frac{1}{q\lambda}$ and $b=\frac{1}{1-q\lambda}$ we get
$a=\frac{p-1}{p-q}>1$, $b=\frac{p-1}{q-1}>1$ and $\frac{1}{a}+\frac{1}{b}=1$. Then, applying Holder's inequality in (\ref{eq:5p2}) with exponents $a$ and $b$ we get 
$$A\leq \big(\int_X (\phi^{q\lambda})^a\big)^{\frac{1}{a}} \big(\int_X (\phi^{q(1-\lambda)})^b\big)^{\frac{1}{b}}.$$
The last stated inequality immediately now gives the right side of (\ref{eq:5p1}).
\end{proof}

\begin{remark}
It is well known that Holder's inequality:
\begin{equation}\label{eq:5p3}
\int_X \phi_1 \phi_2 \; d\mu \leq \big(\int_X \phi_1^a \; d\mu \big)^{\frac{1}{a}} \big(\int_X \phi_2^b \; d\mu \big)^{\frac{1}{b}}
\end{equation}
for the non-negative measurable functions $\phi_1$, $\phi_2$ and exponents $a,b>1$ satisfying $\frac{1}{a}+\frac{1}{b}=1$, becomes equality if and only if there exist non-negative constants $c,d$ with $|c|+|d|>0$
for which the equality
\begin{equation}\label{eq:5p4}
c\phi_1^a(x)=d\phi_2^b(x)
\end{equation}
holds for $\mu$-almost every $x\in X$. Thus by the proof of Lemma 2.1, that is given above we conclude that the equation
\begin{equation}\label{eq:5p5}
A=f^{\frac{p-q}{p-1}}F^{\frac{q-1}{p-1}}
\end{equation}
is true if and only if there exist $c,d$ as above such that $c\phi(x)=d\phi^p(x)$, for $\mu$-almost every $x\in X$. So if we consider the set 
$C_{\phi}=\{x\in X: \phi(x)\neq 0\}$, we must have that $\phi$ should be constant on $C_{\phi}$. We conclude that (\ref{eq:5p5}) is true if and only if there exist a non-negative real number $c$ and a measurable subset $C$ of $X$ for which the equality $\phi=c\chi_C$ is almost everywhere true, where $\chi_C$ denotes the characteristic function of $C$.

Finally we investigate when we do have equality on the first  inequality that is stated in (\ref{eq:5p1}). If $\int_X \phi \; d\mu=f$,  $\int_X\phi^q \; d\mu=A$ then, $f^q=A$ if and only if the function $\phi$ is constant almost everywhere on $X$, and this is a consequence of the condition (\ref{eq:5p4}), assuming that $\phi_1=\phi$, $\phi_2=1$ and $a=q$.
\end{remark}

We now work in the opposite direction. We state the following
\begin{theorem}\label{thm:5p1}
Let $(X,\mu)$ be a non-atomic probability space and suppose that $f,A,F$ are constants satisfying $f^q<A<f^{\frac{p-q}{p-1}}F^{\frac{q-1}{p-1}}$, where $f>0$ and $1<q<p$. Then there exist $a, b>0$ and pairwise disjoint measurable subsets $C_1, C_2$ of $X$, each one of them of positive measure, such that the function defined by $\phi=a\chi_{C_1}+b\chi_{C_2}$, satisfies $\int_X \phi \; d\mu=f$, $\int_X \phi^q \; d\mu=A$ and $\int_X \phi^p \; d\mu=F$.
\end{theorem}

\begin{proof}
By considering non-increasing equimeasurable rearrangements of functions, with the domain of their definition being the interval $(0,1]$, it is enough to prove that, under the conditions that are stated for $f, A, F$ in the statement of the theorem, there exist $a, b>0$ and $k_1, k_2>0$ with $k_1<1$ and $k_2\leq 1-k_1$ for which the function $\phi:(0,1]\rightarrow \mathbb{R}$ defined by $\phi=a\chi_{(0,k_1]}+b\chi_{(k_1,k_1+k_2]}$, satisfies the integral conditions that are stated in the theorem.

Thus we search for $a,b,k_1,k_2$ as above such that the following equations are satisfied:
\begin{equation}\label{eq:5p6}
k_1a+k_2b=f,
\end{equation}
\begin{equation}\label{eq:5p7}
k_1a^q+k_2b^q=A,
\end{equation}
\begin{equation}\label{eq:5p8}
k_1a^p+k_2b^p=F,
\end{equation}
and the inequality $k_2\leq 1-k_1$ holds also true.

We work on this system of equations in the following manner.
First, (\ref{eq:5p6}) is equivalent to $b=\frac{f-k_1a}{k_2}$, thus giving us that $a$ should satisfy $a<\frac{f}{k_1}$. Then (\ref{eq:5p7}) and (\ref{eq:5p8}) respectively become

\begin{equation}\label{eq:5p9}
k_1a^q+\frac{(f-k_1a)^q}{k_2^{q-1}}=A,
\end{equation}
\begin{equation}\label{eq:5p10}
k_1a^p+\frac{(f-k_1a)^p}{k_2^{p-1}}=F.
\end{equation}
In view of (\ref{eq:5p9}) and (\ref{eq:5p10}), $a$ should in fact satisfy
$$0<a<min\{ \frac{f}{k_1},(\frac{A}{k_1})^{\frac{1}{q}}, (\frac{F}{k_1})^{\frac{1}{p}} \}.$$
Now (\ref{eq:5p9}) gives
\begin{equation}\label{eq:5p11}
k_2=\frac{(f-k_1a)^\frac{q}{q-1}}{(A-k_1a^q)^\frac{1}{q-1}}
\end{equation}
and in view of (\ref{eq:5p11}), (\ref{eq:5p10}) becomes
\begin{equation}\label{eq:5p12}
k_1a^p+\frac{(A-k_1a^q)^\frac{p-1}{q-1}}{(f-k_1a)^\frac{p-q}{q-1}}=F.
\end{equation}
Since $k_1,k_2$ should satisfy $k_2\leq 1-k_1$, we must have in view of 
(\ref{eq:5p11}) that
\begin{equation}\label{eq:5p13}
k_1a^q+\frac{(f-k_1a)^q}{(1-k_1)^{q-1}}\leq A.
\end{equation}
From the above we conclude that if we find $k_1 \in (0,1)$ and $a>0$ such that (\ref{eq:5p12}) and (\ref{eq:5p13}) are satisfied, then by defining $k_2$ by (\ref{eq:5p11}) and $b$ by $b=\frac{f-k_1a}{k_2}$, we will have that the initial system of equations (\ref{eq:5p6}), (\ref{eq:5p7}) and (\ref{eq:5p8}), will be satisfied (for these choices of $a,b,k_1,k_2$) with the inequality $k_2\leq 1-k_1$ being also true.

For this purpose we define the following functions
\begin{equation}\label{eq:5p14}
F_{k_1}(a)=k_1a^p+\frac{(A-k_1a^q)^\frac{p-1}{q-1}}{(f-k_1a)^\frac{p-q}{q-1}}.
\end{equation}
\begin{equation}\label{eq:5p15}
G_{k_1}(a)=k_1a^q+\frac{(f-k_1a)^q}{(1-k_1)^{q-1}},
\end{equation}
in the following range of values of $a$:
$0<a<min\{ \frac{f}{k_1},(\frac{A}{k_1})^{\frac{1}{q}}, (\frac{F}{k_1})^{\frac{1}{p}} \}.$
Then (\ref{eq:5p12}) and (\ref{eq:5p13}) are equivalent to
\begin{equation}\label{eq:5p16}
F_{k_1}(a)=F,
\end{equation}
\begin{equation}\label{eq:5p17}
G_{k_1}(a)\leq A,
\end{equation}
and thus we search for $a$ in the above range and $k_1 \in (0,1)$, so that (\ref{eq:5p16}) and (\ref{eq:5p17}) hold true.
We now study the behaviour of the function $G_{k_1}$, where the value $a$ ranges in the interval $[0,\frac{f}{k_1}]$. It's derivative equals
$$G'_{k_1}(a)=qk_1a^{q-1}-qk_1\big(\frac{f-k_1a}{1-k_1} \big)^{q-1},$$
thus $G_{k_1}$ is strictly increasing on the interval $[f,\frac{f}{k_1}]$ and strictly decreasing on $[0,f]$, and thus attains its minimum value at the point $a_0=f$, with $G_{k_1}(f)=f^q<A$. Moreover $G_{k_1}(0)=\frac{f^q}{(1-k_1)^{q-1}}$ and $G_{k_1}(\frac{f}{k_1})=\frac{f^q}{k_1^{q-1}}$.
At this point we consider two cases studying its one of them separately.

{\bf{CASE I}: $f,A$ satisfy $\big(\frac{f^q}{A}\big)^{\frac{1}{q-1}}<\frac{1}{2}$}.

In this case we choose an arbitrary $k_1 \in \big(\big(\frac{f^q}{A}\big)^{\frac{1}{q-1}}, \frac{1}{2} \big)$. Then we immediately see that 
$\frac{f}{k_1}<\big(\frac{A}{k_1}\big)^{\frac{1}{q}}$ and since
$\big(\frac{f^p}{F}\big)^{\frac{1}{p-1}}<\big(\frac{f^q}{A}\big)^{\frac{1}{q-1}}<k_1$, we get that $\frac{f}{k_1}<\big(\frac{F}{k_1}\big)^{\frac{1}{p}}$, thus 
$$min\{ \frac{f}{k_1},(\frac{A}{k_1})^{\frac{1}{q}}, (\frac{F}{k_1})^{\frac{1}{p}} \}=\frac{f}{k_1}.$$
Thus in this case we search for an $a\in (0, \frac{f}{k_1})$ for which
(\ref{eq:5p16}) and (\ref{eq:5p17}) are true. This is possible due to the following reasons: First for each $a\in (0, \frac{f}{k_1})$ it is true that $G_{k_1}(a)<A$, since $G_{k_1}$ has the stated monotonicity properties and $G_{k_1}(\frac{f}{k_1})=\frac{f^q}{k_1^{q-1}}<A$, while
$G_{k_1}(0)=\frac{f^q}{(1-k_1)^{q-1}}<2^{q-1}f^q<A$, by the choice of $k_1$. Thus we just need to find an $a\in (0, \frac{f}{k_1})$ for which 
$F_{k_1}(a)=F$. But it is true that $F_{k_1}(0)= \frac{A^{\frac{p-1}{q-1}}}{f^{\frac{p-q}{q-1}}}<F$ and $\lim_{a\rightarrow\big(\frac{f}{k_1}\big)^-}F_{k_1}(a)=+\infty$, thus by continuity reasons we get the existence of of an $a\in (0, \frac{f}{k_1})$ satisfying the desired properties. The proof is complete in CASE I.

{\bf{CASE II}: $f,A$ satisfy $\big(\frac{f^q}{A}\big)^{\frac{1}{q-1}}\geq\frac{1}{2}$}.

In this case we first choose $k_0$ such that $0<k_0<\frac{1}{2}$ for which $\frac{f^q}{(1-k_0)^{q-1}}<A$. This is possible due to the fact that $f^q<A$. Then note that for every $k_1\in (0,k_0)$ the following inequalities are true:
\begin{equation}\label{eq:5p18}
G_{k_1}(0)=\frac{f^q}{(1-k_1)^{q-1}}<\frac{f^q}{(1-k_0)^{q-1}}<A<\frac{f^q}{k_1^{q-1}}=G_{k_1}(\frac{f}{k_1}),
\end{equation}
where the last inequality in (\ref{eq:5p18}) is true since $\big(\frac{f^q}{A}\big)^{\frac{1}{q-1}}\geq\frac{1}{2}>k_1$.

Now by (\ref{eq:5p18}) and the monotonicity properties that $G_{k_1}$ has, we get that for each $k_1\in (0,k_0)$ there exists unique $\delta_{k_1} \in (0,\frac{f}{k_1})$, such that $G_{k_1}(\delta_{k_1})=A$.
Note also that for every $k_1\in (0,k_0)$, we have $\frac{f^q}{A}>k_1^{q-1}$, thus the following inequality is true for all such $k_1$,
\begin{equation}\label{eq:5p19} 
\big(\frac{A}{k_1}\big)^{\frac{1}{q}}<\frac{f}{k_1}.
\end{equation}

We now restrict further the range of permitted values of $k_1$: Consider $k_0'$ such that $0<k_0'<k_0<\frac{1}{2}$, such that the following inequality is true: $\frac{A^{\frac{p}{q}}}{(k_0')^{\frac{p}{q}-1}}>F$. Then for every $k_1\in (0,k_0')$ we have
 $\frac{A^{\frac{p}{q}}}{(k_1)^{\frac{p}{q}-1}}>\frac{A^{\frac{p}{q}}}{(k_0')^{\frac{p}{q}-1}}>F$, thus the following is true
\begin{equation}\label{eq:5p20} 
\big(\frac{A}{k_1}\big)^{\frac{1}{q}}>\big(\frac{F}{k_1}\big)^{\frac{1}{p}}
\end{equation}

From (\ref{eq:5p19}) and (\ref{eq:5p20}) we have that for every $k_1\in (0,k_0')$ it is true that

$$min\{ \frac{f}{k_1},(\frac{A}{k_1})^{\frac{1}{q}}, (\frac{F}{k_1})^{\frac{1}{p}} \}=(\frac{F}{k_1})^{\frac{1}{p}}.$$
As a consequence we just need to find $k_1\in (0,k_0')$ and 
$a\in (0,(\frac{F}{k_1})^{\frac{1}{p}})$ such that $F_{k_1}(a)=F$ and 
$G_{k_1}(a)\leq A$. 

Now for every $k_1\in (0,k_0')$, $\delta_{k_1}$ satisfies $G_{k_1}(\delta_{k_1})=A$, or equivalently  
\begin{equation}\label{eq:5p21}
k_1\delta_{k_1}^q+\frac{(f-k_1\delta_{k_1})^q}{(1-k_1)^{q-1}}=A
\end{equation}
thus we have $k_1\delta_{k_1}^q<A$, which yields $\delta_{k_1}<(\frac{A}{k_1})^{\frac{1}{q}}$. 

Then by (\ref{eq:5p21}) we have
\begin{equation}\label{eq:5p22}
A-k_1\delta_{k_1}^q=\frac{(f-k_1\delta_{k_1})^q}{(1-k_1)^{q-1}}
\end{equation}
and by (\ref{eq:5p14}) we obtain
\begin{equation}\label{eq:5p23}
F_{k_1}(\delta_{k_1})=k_1\delta_{k_1}^p+\frac{(A-k_1\delta_{k_1}^q)^\frac{p-1}{q-1}}{(f-k_1\delta_{k_1})^\frac{p-q}{q-1}},
\end{equation}
so (\ref{eq:5p22}) and (\ref{eq:5p23}) give
\begin{equation}\label{eq:5p24}
F_{k_1}(\delta_{k_1})=k_1\delta_{k_1}^p+\frac{(f-k_1\delta_{k_1})^p}{(1-k_1)^{p-1}}.
\end{equation}
Since $\delta_{k_1}<(\frac{A}{k_1})^{\frac{1}{q}}$, we get $k_1\delta_{k_1}<k_1^{1-\frac{1}{q}}A \rightarrow 0$, as $k_1\rightarrow 0^+$, so by (\ref{eq:5p22}), letting $k_1\rightarrow 0^+$ we conclude 
$k_1\delta_{k_1}^q \rightarrow A-f^q>0$, and as a consequence we get
$\delta_{k_1}^q \rightarrow +\infty$, and thus $\delta_{k_1} \rightarrow +\infty$, when $k_1\rightarrow 0^+$. But then by the above limit conditions, we have
$k_1\delta_{k_1}^p=k_1\delta_{k_1}^q \delta_{k_1}^{p-q} \rightarrow +\infty$, as $k_1\rightarrow 0^+$.

So it is possible to choose a $k_1\in (0,k_0')$, such that $k_1\delta_{k_1}^p>F$, which by (\ref{eq:5p24}) gives $F_{k_1}(\delta_{k_1})>F$, while also by (\ref{eq:5p14}), we have $F_{k_1}(0)=\frac{A^\frac{p-1}{q-1}}{f^\frac{p-q}{q-1}}<F$. Using the above considerations we conclude that there exists $a\in(0,\delta_{k_1})$, for which $F_{k_1}(a)=F$. But then by (\ref{eq:5p14}) we deduce $k_1a^p<F$, that is 
$a\in (0,(\frac{F}{k_1})^{\frac{1}{p}})$, which ensures us that $a$ lives in the desired interval of definition. Remember now that 
$G_{k_1}(\delta_{k_1})=A$, while $0<a<\delta_{k_1}$ and since (\ref{eq:5p18}) is true, we have
$$G_{k_1}(0)=\frac{f^q}{(1-k_1)^{q-1}}<A<\frac{f^q}{k_1^{q-1}}=G_{k_1}(\frac{f}{k_1}).$$
Then, by the monotonicity properties that $G_{k_1}$ has, and by the choices of $\delta_{k_1}$ and $a$ we immediately get $G_{k_1}(a)<A$, so the desired conditions for $k_1$ and $a$ hold true and the proof is complete in CASE II also.

\end{proof}
\section{Sharpness of Theorem 1.1 - connections with $B_{p,q}^{\mc T}(f,A,F)$ } \label{sec:6}

As we saw in Section 5 the set of triples $(f,A,F)$, for which there exists $\phi$, as in the definition of (\ref{eq:1p8}) is realized by the conditions
\begin{equation}\label{eq:6p1}
f^q\leq A \leq f^{\frac{p-q}{p-1}} F^{\frac{q-1}{p-1}}. 
\end{equation}
Assume now that $f,F$, are given such that $0<f^p\leq F$. Define 
$k_0=\big(\frac{f^p}{F}\big)^{\frac{1}{p-1}}$ and the function $g$ on the interval $[k_0,1]$, by $g(k)=k^{q-1} H_q(\omega_p(\frac{f^p}{k^{p-1}F}))$.
Notice that $g$ is well defined and is positive on the interval $[k_0,1]$, since $\frac{f^p}{k^{p-1}F} \in (\frac{f^p}{F},1] \subseteq (0,1] $
and $\omega_p(\frac{f^p}{k^{p-1}F}) \in [1,\frac{p}{p-1}) \subseteq [1,\frac{q}{q-1})$, for all $k\in [k_0,1]$. We will need the following simple

\begin{lemma} \label{lem:6p1}
The function $g:[k_0,1] \rightarrow \mathbb{R}$ is strictly increasing and continuous.
\end{lemma}

\begin{proof}
We differentiate $g$ and after some simple calculations we see that $g'(k)$ is of the same sign as $H_q(t)-\frac{q}{p}\frac{1}{t^{p-q}}H_p(t)$, where $t=\omega_p(\frac{f^p}{k^{p-1}F})$. Thus $g'(k)$ has the same sign with the constant $\frac{q}{p}-\frac{q-1}{p-1}$ which is positive. The continuity of $g$ is obvious.
\end{proof}
Consider now the case where $f,A,F$, additionally to (\ref{eq:6p1}), satisfy the condition 
$\omega_q(\frac{f^q}{A})> \omega_p(\frac{f^p}{F})$, which is equivalent to
\begin{equation}\label{eq:6p2}
\frac{f^q}{H_q(\omega_p(\frac{f^p}{F}))}< A \leq f^{\frac{p-q}{p-1}} F^{\frac{q-1}{p-1}}. 
\end{equation}
Then, since $g(k_0)=(\frac{f^p}{F})^{\frac{q-1}{p-1}}$, we conclude by (\ref{eq:6p2}), that $g(k_0) \leq \frac{f^q}{A}<g(1)$, thus by Lemma 5.1 we get that there exists (for each such $A$) unique $k\in [k_0,1)$ for which $g(k)=\frac{f^q}{A}$, or equivalently
\begin{equation}\label{eq:6p3}
\omega_q(\frac{f^q}{k^{q-1}A})=\omega_p(\frac{f^p}{k^{p-1}F})=:\epsilon_0 .
\end{equation}

From the above we conclude that as $A$ ranges according to (\ref{eq:6p2}), the respective value of $k$ ranges in the whole interval $[k_0,1)$ and conversely, if for each $k\in [k_0,1)$ we define $A=\frac{f^q}{g(k)}$  then $A$ ranges over the whole interval that is described by (\ref{eq:6p2}). Thus there is an one to one correspondence between $k$ and $A$ on the respective half open intervals.

At this point we remark also that if $A$ where such that we had equality in the first inequality in (\ref{eq:6p2}), then by the results of \cite{DN}, the value of $B_{p,q}^{\mc T}(f,A,F)$ will be determined by
$B_p^{\mc T}(f,F,f,1)$, that is, in this case, we will have 

$$B_{p,q}^{\mc T}(f,A,F)=F\omega_p(\frac{f^p}{F})^p.$$

Now let $A$ and $k$ be as in (\ref{eq:6p2}) and (\ref{eq:6p3}) respectively. We define the following non-increasing function on the interval $(0,1]$: $\lambda(t)=ct^{-1+\frac{1}{\epsilon_0}}$, for $t\in (0,k]$, and $\lambda(t)=0$ for $t\in (k,1]$, where the constant c is such that the integral inequality $\int_0^1 \lambda=f$ holds. Solving this last equation on $c$ we deduce $c=\frac{f}{\epsilon_0k^{\frac{1}{\epsilon_0}}}$. Substituting this value to the definition of the function $\lambda$, and using (\ref{eq:6p3}) we easily get that the integral conditions $\int_0^1 \lambda^q=A$ and $\int_0^1 \lambda^p=F$ are true. Additionally we easily see that the equality
\begin{equation}\label{eq:6p4}
\frac{1}{t}\int_0^t \lambda=\epsilon_0 \lambda(t), t\in (0,k],
\end{equation}
is true.
Note also that if we define 
$\delta=\delta_{k'}=\left( \frac{\int_0^{k'}\left(\frac{1}{t}\int_0^t \lambda\right)^p\; dt}{\int_0^{k'} \lambda^p} \right)^{\frac{1}{p}}$, for $0<k'<k$,  we see by (\ref{eq:6p4}), that $\delta=\epsilon_0=\omega_p(\frac{f^p}{k^{p-1}F})\leq \omega_p(\frac{f^p}{F})$, that is (\ref{eq:4p5}) is true for this choice of the function $\lambda$, for every $k'$ such that $0<k'<k$, while also $\int_{k'}^1\lambda>0$ for every such $k'$. 

We prove now that we have equality on inequality (\ref{eq:4p10}) in the limit, as $k'\rightarrow k^{-}$, where we replace $k$ by $k'$ and $h$ by $\lambda$. More precisely the inequality

\begin{equation}\label{eq:6p5}
\int_{k'}^1 \left( \frac{1}{t} \int_0^t \lambda \right)^p \;dt \leq \int_{k'}^1 h^p \omega_p \left(\frac{f^p-H_{p}(\delta_{k'}) \int_0^{k'} \lambda^p }{\int_{k'}^1 \lambda^p}\right)^p,
\end{equation} 
is true for every $k'\in (0,k)$ by the conditions that $\lambda$ satisfies, as mentioned above and by the results of Section 4.

Now if we let $k'\rightarrow k^{-}$, the left side of (\ref{eq:6p5}) tends to $$\int_k^1 \left( \frac{1}{t} \int_0^t \lambda \right)^p \;dt=\int_k^1f^p(\frac{1}{t})^p \;dt=\frac{1}{p-1}(\frac{f^p}{k^{p-1}}-f^p)=:\alpha(k).$$

Letting $k'\rightarrow k^{-}$ on the right side of (\ref{eq:6p5}) we immediately see that it tends to the limit
$$\lim_{y\rightarrow 0^+}y\omega_p(\frac{-(p-1)\alpha(k)}{y})^p= \alpha(k).$$ 

Thus by using the results of Section 4 we get the following 
\begin{corollary} \label{cor:6p1}
If the variables $f,F,k$ satisfy $f^p<F$ and 
$k\in  [(\frac{f^p}{F})^{\frac{1}{p-1}},1)$, then there exists a sequence of non-negative functions $(\phi_n)$ defined on $(X,\mu)$ and satisfying
$\int_X \phi_n d\mu=f$, $\int_X \phi^p_n d\mu=F$, a sequence of measurable sets $(K_n)$ with $\mu(K_n) \rightarrow k$, and two sequences $(\beta_n)$, $(\gamma_n)$ such that $\beta_n\geq \gamma_n >0$, in a way that if we replace $\phi_n, K_n,\beta_n, \gamma_n$ in (\ref{eq:1p9}) we get equality in the limit. Moreover we can achieve every $\phi_n$ to satisfy $\int_X \phi^q_n d\mu=A$, where $A$ is given by (\ref{eq:6p3}).
\end{corollary}

\begin{corollary} \label{cor:6p2}
$B_{p,q}^{\mc T}(f,A,F)\geq \omega_p(\frac{f^p}{k^{p-1}F})^pF + \frac{1}{p-1}(\frac{f^p}{k^{p-1}}-f^p)$, for every $(f,A,F)$ satisfying (\ref{eq:6p2}) where $k$ is given by (\ref{eq:6p3}).
\end{corollary}

\begin{proof}
Immediate, since the non-increasing function $\lambda$ satisfies the integral conditions on $(0,1]$ corresponding to (\ref{eq:1p8}) and thus by Theorem 2.1 in Section 2, we get 
\begin{equation}\label{eq:6p6}
B_{p,q}^{\mc T}(f,A,F)\geq \int_0^1 \left( \frac{1}{t} \int_0^t \lambda \right)^p \;dt, 
\end{equation}
while the right side of (\ref{eq:6p6}) equals 
$$\int_0^k \left( \frac{1}{t} \int_0^t \lambda \right)^p \;dt+\int_k^1 \left( \frac{1}{t} \int_0^t \lambda \right)^p \;dt= \omega_p(\frac{f^p}{k^{p-1}F})^pF + \frac{1}{p-1}(\frac{f^p}{k^{p-1}}-f^p) $$
\end{proof}
At this point we remark the following

{\bf{Conjecture :}} In Corollary 6.2 we in fact have equality.  
Note also that the above mentioned conjecture is true on the extreme case $A=\frac{f^q}{H_q(\omega_p(\frac{f^p}{F}))}$ (equivalently when $k=1$), and this is because of the results in \cite{DN}.

\section{A lower bound for $B_{p,q}(f,A,F)$ when $(f,A,F)$ satisfies $\omega_q \left ( \frac{f^q}{A} \right ) < \omega_p \left ( \frac{f^p}{F} \right )$} \label{sec:7}

In this section we study the rest of the domain of the definition of $B_{p,q}(f,A,F)$,so that we assume that $f,A,F$ satisfy
\begin{equation} \label{eq:7p1}
   f^q < A < \frac{f^q}{H_q \left ( \omega_p \left ( \frac{f^p}{F} \right ) \right )} 
\end{equation}
Now for an arbitrary $k \in (0,1)$ we define the function
$g_k : [0,1] \to \mathbb{R}^+$ by the following way:
\[
g_k(t) = \begin{cases}
    A_1 t^{-1+\frac{1}{a}}\ , \ t\in (0,k] \\
    c \ , \ t \in (k,1]
\end{cases}
\]
where the constants $A_1,a,c$ obey the following rules:
\begin{align}
    &c = \frac{f-B_0}{1-k} \ , \label{eq:7p2} \\
    &A_1 = \frac{B_0 k^{-1/a}}{a} \ , \label{eq:7p3} \\
    &a = \omega_p(z_0) \ , \label{eq:7p4}
\end{align}
where $z_0$ is given by 
\begin{equation} \label{eq:7p5}
z_0 = \frac{B_0^p}{k^{p-1} \left ( F - \frac{(f-B_0)^p}{(1-k)^{p-1}} \right )}
\end{equation}
and $B_0$ satisfies
\begin{equation} \label{eq:7p6}
    \omega_p(z_0) = \frac{B_0}{k} \frac{1-k}{f-B_0} .
\end{equation}
Such a choice of $B_0$ is possible (for details see \cite{M}, page 321-322, Lemma 7). Then according to the results in \cite{M} , $B_0$ should also satisfy
\begin{equation} \label{eq:7p7}
    \frac{f(1-k)}{f-B_0} = \omega_{p,k} \left ( \frac{f^p}{p} \right ) 
\end{equation}
For the definition of $\omega_{p,k} \left ( \frac{f^p}{p} \right ) $,see Lemma $7(i)$, page 321 in \cite{M}.

\vspace{5mm}

\noindent Now in \cite{N5} it is proved that the function $g_k$ that is constructed above satisfies
\begin{align}
    &\int_0^1 g_k = f \ , \ \int_0^1 g_k^p = F \ , \ \forall k \in (0,1) \label{eq:7p8} \\
    &\int_0^k g_k = B_0 \ , \label{eq:7p9} \\
    &\frac{1}{t} \int_0^t g_k = a g_k(t) \ , \ \forall t \in (0,1) \ , \ \forall k \in (0,1) \ , \label{eq:7p10} \\
    &g_k(k) = c = \frac{f-B_0}{1-k} \ , \label{eq:7p11}
\end{align}
Note that $B_0$ depends on $f,F,k$ and that (\ref{eq:7p10}) implies that $g_k$ is continuous and non-increasing in $(0,1]$.

\vspace{5mm}

\noindent We now evaluate the $L^q$-integral of $g_k,$ for any $k \in (0,1),$ as follows:
\begin{align*}
    A_k &:= \int_0^1 g_k^q = \int_0^k g_k^q + \int_k^1 g_k^q \\
    &= \int_0^k g_k^q + (1-k)c^q \\
    &= \frac{(f-B_0)^q}{(1-k)^{q-1}} + \int_0^k A_1^q t^{-q+\frac{q}{a}}dt \\
    &= \frac{(f-B_0)^q}{(1-k)^{q-1}} + \frac{B_0^q k^{-q/a}}{a^q} \cdot \frac{k^{-q+q/a+1}}{(-q+q/a+1)} \\
    &= \frac{(f-B_0)^q}{(1-k)^{q-1}} + \frac{B_0^q}{k^{q-1}} \frac{1}{qa^{q-1}-(q-1)a^q} \\
    &= \frac{(f-B_0)^q}{(1-k)^{q-1}} + \frac{B_0^q}{k^{q-1}} \frac{1}{H_q(a)} \\
    &= \frac{(f-B_0)^q}{(1-k)^{q-1}} + \frac{B_0^q}{k^{q-1}} \frac{1}{H_q(\omega_p(z_0))} \\
    &= (1-k) \left ( \frac{f-B_0}{1-k} \right )^q + \frac{B_0^q}{k^{q-1}} \frac{1}{H_q(\omega_p(z_0))} \\
    &= (1-k) \left ( \frac{B_0}{k} \right )^q \frac{1}{\omega_p(z_o)^q} + k \left ( \frac{B_0}{k} \right )^q \frac{1}{H_q(\omega_p(z_0))} \\
    &= \left ( \frac{B_0}{k} \right )^q \left [(1-k) \frac{1}{\omega_p(z_0)^q} + k \frac{1}{H_q(\omega_p(z_0))}  \right ] \ , \numberthis \label{eq:7p12}
\end{align*}
where we have used the definition of $g_k$, and relations \eqref{eq:7p2}, \eqref{eq:7p3}, \eqref{eq:7p4} and \eqref{eq:7p6}.

\vspace{5mm}

\noindent We prove the following \\

\noindent \underline{\text{Claim} \ 1}:
\[
 \quad \lim \limits_{k \to 1^-} A_k = \frac{f^q}{H_q \left (\omega_p \left ( \frac{f^p}{F} \right ) \right )}
\]
\underline{Proof of claim 1}: \\

\noindent Note that since $g_k$ is non-increasing on $(0,1]$, we have
\[
\frac{f-B_0}{1-k} = \frac{1}{1-k} \int_k^1 g_k \leq \int_0^1 g_k = f \ ,
\]
which gives $f-B_0\leq f(1-k)$, thus letting $k \to 1^-$ we get
\[
\lim \limits_{k \to 1^-} B_0 = f .
\]
Additionally
\[
\frac{(f-B_0)^q}{(1-k)^{q-1}} = (1-k) \left ( \frac{f-B_0}{1-k} \right )^q \leq (1-k) f^q ,
\]
so that
\[
\lim \limits_{k \to 1^-} \frac{(f-B_0)^q}{(1-k)^{q-1}} = 0 ,
\]
and analogously $$\lim \limits_{k \to 1^-} \frac{(f-B_0)^p}{(1-k)^{p-1}} = 0 . $$ Thus we get
\[
z_0 := \frac{B_0^p}{k^{p-1}\left ( F - \frac{(f-B_0)^p}{(1-k)^{p-1}} \right ) } \rightarrow \frac{f^p}{F} , \quad \text{as} \quad k \to 1^-
\]
and by the chain of equalities in \eqref{eq:7p12} we get that
\[
\lim \limits_{k \to 1^-} A_k = \frac{f^q}{H_q \left ( \omega_p \left ( \frac{f^p}{F} \right ) \right )} \ , 
\]
so that claim $1$ is proved. \\

\noindent \underline{\text{Claim} \ 2}:
\[  
    \lim \limits_{k \to 0^+} A_k = f^q
\]

\noindent \underline{Proof of claim 2}: \\

\noindent By \eqref{eq:7p6} we get
\[
B_0 = k \ \frac{f-B_0}{1-k} \ \omega_p(z_0) \ ,
\]
so that $B_0 \to 0$ , as $k \to 0^+$ , since the factor $\frac{f-B_0}{1-k}$ , remains bounded as $k \to 0^+$ and $\omega_p(z_0) \in \left [ 1 , \frac{p}{p-1} \right ]$ , for any $k \in (0,1)$. \\

\noindent Moreover, for any $k \in (0,1)$
\begin{align*}
&\frac{B_0^p}{k^p} = \left ( \frac{1}{k} \int_0^k g_k \right )^p \leq \frac{1}{k} \int_0^k g_k^p \leq  \frac{1}{k}\int_0^1 g_k^p = \frac{1}{k} F \\
&\Rightarrow \frac{B_0^p}{k^{p-1}} \leq F \ , \ \forall k \in (0,1).
\end{align*}
Thus $B_0$ satisfies $B_0 \leq k^{1-\frac{1}{p}} F^{1/p} \ ,$ and consequently we have
\[
\frac{B_0^q}{k^{q-1}} \leq \frac{k^{q-\frac{q}{p}} F^{q/p}}{k^{q-1}} = k^{1-\frac{q}{p}} F^{q/p} \to 0 \ ,
\]
as $k \to 0^+$ , since $q \in (1, p)$ . \\

\noindent Moreover, since $1 \leq \omega_p (z_0) \leq \frac{p}{p-1} ,$ for any $k \in (0,1) ,$ we have
\[
\frac{1}{H_q(\omega_p(z_0))} \leq \frac{1}{H_q \left ( \frac{p}{p-1} \right )} .
\]
By the comments above we immediately get that $$\frac{B_0^q}{k^{q-1}} \ \frac{1}{H_q(\omega_p(z_0))} \to 0 ,$$ 
as $k \to 0^+ .$ Thus by \eqref{eq:7p12} we have that
\[
\lim \limits_{k \to 0^+} A_k = \lim \limits_{k \to 0^+} \frac{(f-B_0)^q}{(1-k)^{q-1}} = f^q \ ,
\]
and claim 2 is proved. \\

\noindent \underline{\text{Claim} \ 3}:
\[
\text{The function} \ k \mapsto A_k \ , \text{defined on} \ k \in (0,1) \ \text{is continuous}.
\]
\noindent \underline{Proof of claim 3}: \\

\noindent We write $B_0 = B_0(k)$ and $Z_0 = Z_0(k)$ , denoting the dependence of $B_0$ and $z_0$ on $k \in (0,1)$. If we prove that $B_0$ is a continuous function on $k \in (0,1)$ , then by \eqref{eq:7p5} and \eqref{eq:7p12} , the claim follows. \\

\noindent For this purpose, it is enough to prove that, for any fixed $k \in (0,1)$ and every sequence $(k_n)_{n \in \mathbb{N}} \subseteq (0,1)$ for whick $k_n \to k$, there exists a subsequence: $(k_{n_m})_{m \in \mathbb{N}}$ for which 
\begin{equation} \label{eq:7p13}
    \lim \limits_{n} B_0(k_{n_m}) = B_0(k) .
\end{equation}
But for any $n \in \mathbb{N}$ , $B_0(k_n)$ satisfies (using \eqref{eq:7p6}) the following equality:
\begin{align*}
    \omega_p(z_0(k_n)) &= \omega_p \left ( \frac{B_0(k_n)^p}{k_n^{p-1} \left ( F - \frac{(1 - B_0(k_n))^p}{(1-k_n)^{p-1}} \right )} \right ) \\
    &= \frac{B_0(k_n)}{k_n} \cdot \frac{1-k_n}{1-B_0(k_n)} \numberthis \label{eq:7p14} .
\end{align*}
Also, $(B_0(k_n))_{n \in \mathbb{N}}$ is a bounded sequence since
\[
0 \leq B_0(k_n)_{n \in \mathbb{N}} = \int_0^{k_n} g_{k_n} \leq \int_0^1 g_{k_n} = f ,
\]
so that there exists a subsequence: $(B_0(k_{n_m}))_{m \in \mathbb{N}}$ for which
\[
B_0(k_{n_m}) \to \lambda , \quad \text{for some} \ \lambda \in [0,f] .
\]

\noindent Note also that $\lambda$ should satisfy $\lambda \neq f$, since otherwise, if we let $n=n_m \to \infty$ in \eqref{eq:7p14} we would get that the right hand side will tend to infinity, while the left hand side is bounded. Thus $\lambda \in [0,f).$ \\

\noindent Then, taking limits in \eqref{eq:7p14} (replacing first $n$ by $n_m$) we get
\begin{equation} \label{eq:7p15}
    \omega_p(z) = \frac{\lambda}{k} \ \frac{1-k}{f-\lambda} ,
\end{equation}
where
\begin{equation} \label{eq:7p16}
    z = \frac{\lambda^p}{k^{p-1} \left ( f - \frac{(f-\lambda)^p}{(1-k)^{p-1}}\right )} , 
\end{equation}
Then by the proof of Lemma $7$, in \cite{M} (pages 321-322) we see that $\lambda$ should satisfy:
\begin{equation} \label{eq:7p17}
    \frac{f(1-k)}{f-\lambda} = \omega_{p,k} \left ( \frac{f^p}{F} \right ) ,
\end{equation}
while the definition of $B_0$ implies also that
\begin{equation} \label{eq:7p18}
    \frac{f(1-k)}{f-B_0} = \omega_{p.k} \left ( \frac{f^p}{F} \right ) ,
\end{equation}
By \eqref{eq:7p17} and \eqref{eq:7p18} we immediately get that $\lambda = B_0(k)$, thus yielding $$\lim \limits_{n} B_0(k_{n_m})=B_0(k)$$ and \eqref{eq:7p13} follows, thus giving the proof of claim 3. \\

\noindent By claims 1, 2, 3 above we get that, for any $$ A \in \left ( f^q, \frac{f^q}{H_q \left ( \omega_p \left ( \frac{f^p}{F} \right ) \right )} \right ) ,$$
there exists $k \in (0,1)$ for which: $A_k = A$. As a consequence we have proved the following

\begin{corollary} 
  If the variables $f,A,F$ satisfy
    \[
        f^q < A < \frac{f^q}{ H_q \left ( \omega_p \left (\frac{f^p}{p} \right ) \right )} ,
    \]
    then there exists $k \in (0,1)$ for which the function $g_k$ satisfies
    \[
    \int_0^1 g_k = f \ , \quad \int_0^1 g_k^q = A \quad \text{and} \quad \int_0^1 g_k^p = F \ ,
    \]
    so that
    \[
    B_{p,q} (f,A,F) \geq \int_0^1 \left ( \frac{1}{t} \int_0^t g_k \right )^p dt . \quad \qed
    \]
\end{corollary}

\vspace{10pt}
\noindent Nikolidakis Eleftherios\\
Associate Professor\\
Department of Mathematics \\
Panepistimioupolis, University of Ioannina, 45110\\
Greece\\
E-mail address: enikolid@uoi.gr

\end{document}